\documentclass[11pt]{article}
\usepackage[left=1in,right=1in,bottom=1in,top=1in]{geometry}
\usepackage{amsmath}
\usepackage{amsfonts}
\usepackage{amssymb}
\usepackage{amsthm}
\usepackage{young}
\usepackage[vcentermath]{youngtab}
\usepackage{tikz}
\usetikzlibrary{matrix}

\usepackage{algorithm}
\usepackage[noend]{algpseudocode}
\usepackage{amsmath, amssymb, amsthm}
\usepackage{blindtext}
\usepackage[pdftex,bookmarks=true]{hyperref}
\usepackage{cleveref}
\usepackage{comment}
\usepackage{enumerate}
\usepackage{fullpage}
\usepackage{mathtools}
\usepackage{subcaption}

\makeatletter
\newtheorem*{rep@theorem}{\rep@title}
\newcommand{\newreptheorem}[2]{%
\newenvironment{rep#1}[1]{%
 \def\rep@title{#2 \ref{##1}}%
 \begin{rep@theorem}}%
 {\end{rep@theorem}}}
\makeatother

\parskip \medskipamount
\newcommand{\Z}{\mathbb Z}
\newcommand{\R}{\mathbb R}

\newcommand{\ep}{\varepsilon}

\newtheorem{theorem}{Theorem}[section]
\newtheorem{lemma}[theorem]{Lemma}
\newreptheorem{theorem}{Theorem}

\newtheorem{definition}[theorem]{Definition}
\newtheorem{proposition}[theorem]{Proposition}

\newcommand{\pr}[1]{\mathbb{P}\!\left(#1\right)}
\newcommand{\expect}[1]{\mathbb{E}\!\left[#1\right]}

\title{On the mixing time of the Diaconis--Gangolli random walk on contingency tables over $\mathbb{Z}/ q \mathbb{Z}$}

\author{Evita Nestoridi \thanks{ Department of Mathematics, Princeton University, USA,  emails: exn@princeton.edu, onguyen@princeton.edu.} \and Oanh Nguyen \footnotemark[1]}

\include{psfig}
\usepackage{mathrsfs}\usetikzlibrary{arrows}
\begin{document}
	\date{}
\maketitle

\begin{abstract}
The Diaconis--Gangolli random walk is an algorithm that generates an almost uniform random graph with prescribed degrees. In this paper, we study the mixing time of the Diaconis--Gangolli random walk restricted on $n\times n$ contingency tables over $\Z/q\Z$. We prove that the random walk exhibits cutoff at $\frac{n^2}{4(1- \cos{\frac{2 \pi}{q}})} \log n, $ when $\log q=o\left (\frac{\sqrt{\log n}}{\log \log n}\right )$. 
\end{abstract}

\section{Introduction}
Random graphs are an important object of study in Combinatorics, Computer Science, and Probability.
An $(n, n)$ bipartite graph is a graph with a left vertex set $V=\{ v_1, \dots, v_n\}$ and a right vertex set $U=\{u_1, \dots, u_n\}$ and the only edges are the edges that connect a vertex in $V$ with a vertex in $U$. Consider the problem of generating a random bipartite graph with prescribed degrees. How do we choose uniformly at random such an $(n, n)$ bipartite graph? 

Diaconis and Gangolli \cite{DiaGan} proposed the following randomized algorithm. Start with a bipartite  graph $G_0$ that has the desired vertex degrees at time $t=0$. At time $t$, from the graph $G_{t-1}$, choose two pairs of vertices $v_i \neq v_j$ in $V$ and $u_k \neq u_l$ in $U$ uniformly at random. We delete the already existing edges between $v_i, v_j$ and $u_k , u_l$ and we draw the remaining edges between them to get $G_t'$. If the result is not a graph with the prescribed degrees, then $G_t$ is set to be $G_{t-1}$. Otherwise, $G_t := G_t'$. In the language of random graphs, such a procedure is also known as the \textit{simple switching} method developed by McKay and Wormald and many others (see, for example, the survey \cite{wormald1999models} and the references therein). 

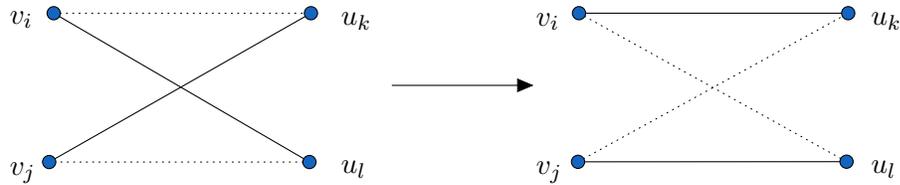
\definecolor{rvwvcq}{rgb}{0.08235294117647059,0.396078431372549,0.7529411764705882}\begin{tikzpicture}[line cap=round,line join=round,>=triangle 45,x=1cm,y=1cm]\clip(-8.279203169356013,-0.7575173213097165) rectangle (8.642991526045225,2.688957866157207);

\draw (-5.4,1.9302374292040572) node[anchor=north west] {$v_i$};

\draw [line width=0.4pt] (-4.68,1.77)-- (-1.28,-0.21);\draw [line width=0.4pt,dotted] (-1.28,-0.21)-- (-4.74,-0.21);\draw [line width=0.4pt] (-4.74,-0.21)-- (-1.26,1.77);\draw [line width=0.4pt,dotted] (-4.68,1.77)-- (-1.26,1.77);

\draw (-5.4,-0.0781401980248697) node[anchor=north west] {$v_j$};

\draw (-0.9918996733595258,1.9153605578912505) node[anchor=north west] {$u_k$};

\draw (-0.9918996733595258,-0.05830436960779387) node[anchor=north west] {$u_l$};

\draw [->,line width=0.4pt] (-0.18,0.81) -- (1.7,0.81);

\draw (1.6,1.9054426436827125) node[anchor=north west] {$v_i$};\draw [line width=0.4pt,dotted] (2.3035379501227378,1.7702191413494348)-- (5.860342937174081,-0.20531034550720026);\draw [line width=0.4pt] (5.860342937174081,-0.20531034550720026)-- (2.2906481364340694,-0.20715652734024392);\draw [line width=0.4pt,dotted] (2.2906481364340694,-0.20715652734024392)-- (5.882712779859938,1.7702191413494348);\draw [line width=0.4pt] (2.3035379501227378,1.7702191413494348)-- (5.882712779859938,1.7702191413494348);

\draw (1.6,-0.07318124092060074) node[anchor=north west] {$v_j$};
\draw (6.050863436865267,1.9153605578912505) node[anchor=north west] {$u_k$};
\draw (6.050863436865267,-0.05334541250352492) node[anchor=north west] {$u_l$};

\begin{scriptsize}\draw [fill=rvwvcq] (-4.68,1.77) circle (2.5pt);\draw [fill=rvwvcq] (-1.28,-0.21) circle (2.5pt);\draw [fill=rvwvcq] (-4.74,-0.21) circle (2.5pt);\draw [fill=rvwvcq] (-1.26,1.77) circle (2.5pt);\draw [fill=rvwvcq] (2.3035379501227378,1.7702191413494348) circle (2.5pt);\draw [fill=rvwvcq] (5.860342937174081,-0.20531034550720026) circle (2.5pt);\draw [fill=rvwvcq] (2.2906481364340694,-0.20715652734024392) circle (2.5pt);\draw [fill=rvwvcq] (5.882712779859938,1.7702191413494348) circle (2.5pt);\end{scriptsize}\end{tikzpicture}

Two main questions concerning random bipartite graph with prescribed degrees are: what is their number and how long does the Diaconis--Gangolli algorithm take to produce such a random bipartite graph? 

In this paper we consider a more specialized version of the Diaconis--Gangolli algorithm. To start with a simple example, on $\Z/2\Z$, consider the problem of generating a random bipartite graph in which every vertex has an even degree. At time $t=0$, we start with the empty graph $G_0$. At time $t$, the Diaconis--Gangolli algorithm suggests that from the graph $G_{t-1}$, choose two pairs of vertices $v_i \neq v_j$ in $V$ and $u_k \neq u_l$ in $U$ uniformly at random. Then, replace the subgraph induced by restricting $G_{t-1}$ on these four vertices by its complement.

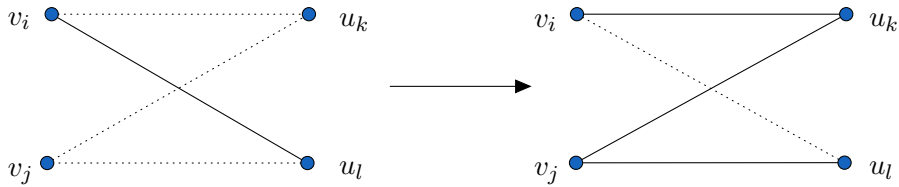
\definecolor{rvwvcq}{rgb}{0.08235294117647059,0.396078431372549,0.7529411764705882}\begin{tikzpicture}[line cap=round,line join=round,>=triangle 45,x=1cm,y=1cm]\clip(-8.279203169356013,-0.7575173213097165) rectangle (8.642991526045225,2.688957866157207);

\draw (-5.4,1.9302374292040572) node[anchor=north west] {$v_i$};

\draw [line width=0.4pt] (-4.68,1.77)-- (-1.28,-0.21);\draw [line width=0.4pt,dotted] (-1.28,-0.21)-- (-4.74,-0.21);\draw [line width=0.4pt,dotted] (-4.74,-0.21)-- (-1.26,1.77);\draw [line width=0.4pt,dotted] (-4.68,1.77)-- (-1.26,1.77);

\draw (-5.4,-0.0781401980248697) node[anchor=north west] {$v_j$};

\draw (-0.9918996733595258,1.9153605578912505) node[anchor=north west] {$u_k$};

\draw (-0.9918996733595258,-0.05830436960779387) node[anchor=north west] {$u_l$};

\draw [->,line width=0.4pt] (-0.18,0.81) -- (1.7,0.81);

\draw (1.6,1.9054426436827125) node[anchor=north west] {$v_i$};\draw [line width=0.4pt,dotted] (2.3035379501227378,1.7702191413494348)-- (5.860342937174081,-0.20531034550720026);\draw [line width=0.4pt] (5.860342937174081,-0.20531034550720026)-- (2.2906481364340694,-0.20715652734024392);\draw [line width=0.4pt] (2.2906481364340694,-0.20715652734024392)-- (5.882712779859938,1.7702191413494348);\draw [line width=0.4pt] (2.3035379501227378,1.7702191413494348)-- (5.882712779859938,1.7702191413494348);

\draw (1.6,-0.07318124092060074) node[anchor=north west] {$v_j$};
\draw (6.050863436865267,1.9153605578912505) node[anchor=north west] {$u_k$};
\draw (6.050863436865267,-0.05334541250352492) node[anchor=north west] {$u_l$};

\begin{scriptsize}\draw [fill=rvwvcq] (-4.68,1.77) circle (2.5pt);\draw [fill=rvwvcq] (-1.28,-0.21) circle (2.5pt);\draw [fill=rvwvcq] (-4.74,-0.21) circle (2.5pt);\draw [fill=rvwvcq] (-1.26,1.77) circle (2.5pt);\draw [fill=rvwvcq] (2.3035379501227378,1.7702191413494348) circle (2.5pt);\draw [fill=rvwvcq] (5.860342937174081,-0.20531034550720026) circle (2.5pt);\draw [fill=rvwvcq] (2.2906481364340694,-0.20715652734024392) circle (2.5pt);\draw [fill=rvwvcq] (5.882712779859938,1.7702191413494348) circle (2.5pt);\end{scriptsize}
\end{tikzpicture}

Our main theorem below asserts that at time $t \ge \frac{n^2}{8} \log n + 10 c n^2 \sqrt{\log n}\log \log n$, the (random) graph $G_t$ is distributed almost uniformly: the $L^{1}$ distance between the distribution of $G_t$ and the uniform distribution on the bipartite graphs whose all degrees are even is at most $e^{-c}$.

Similarly, one can use the same algorithm to generate a uniform bipartite graph with prescribed degree parities, say, the vertices in $A\subset V$ and $B\subset U$ have odd degrees and the rest have even degrees. The only difference is in choosing the starting point $G_0$. Instead of starting at the empty graph, one can start at any graph that satisfies the prescribed degree parities.  

Associate each bipartite graph with its adjacency matrix $A,$ which is an $n\times n$ matrix whose $(i, j)$ entry is $1$ if $(v_i, u_j)$ is an edge and 0 otherwise. The requirement that each vertex has a specific degree parity means that we require $A$ to have given row sums and column sums in $\mathbb{Z}/ 2\mathbb{Z}$.

More generally, we consider matrices with prescribed row and column sums in $\Z/q\Z$ for positive integers $q$.
\begin{definition}
A contingency table over $\mathbb{Z}/ q \mathbb{Z}$ is an $n\times n$ matrix with entries in $\mathbb{Z}/ q \mathbb{Z}$, with fixed row sums $(r_1, \ldots, r_n) \in (\mathbb{Z}/ q \mathbb{Z})^n$ and fixed column sums $(c_1, \ldots, c_n)  \in (\mathbb{Z}/ q \mathbb{Z})^n$. Note that for the existence of such matrices, we need that $\sum_{i=1}^n r_i = \sum_{j=1}^n c_j$.
\end{definition}

The Diaconis--Gangolli random walk on $n \times n$ contingency tables over $\mathbb{Z}/ q \mathbb{Z}$ goes as follows. Pick two distinct rows and two distinct columns at random. Then, we look at the equilateral that these rows and columns form. We flip a fair coin. If heads we add
$$T= \left( \begin{matrix}
& -1 & 1 \\
&1 & -1 
\end{matrix}   \right)  
$$
to the corners of the equilateral. If tails, then we add $-T$
to the corners of the equilateral.

Let $\Omega$ be the set of all contingency tables with the prescribed degrees in $\Z/q\Z$ and let $A_t$ be the contingency table after $t$ steps of the process. For $A,B \in \Omega$, let $P^t_{A}(B)$ be the probability of moving from $A$ to $B$ after $t $ steps. The sequence of probability measures $P^t_{A}$ converges to the uniform measure $\pi$ on $\Omega,$
as $t \rightarrow \infty$, with respect to the total variation distance
\begin{equation}\label{eq:distance:1}
d(t) := \max_{A \in \Omega} \Vert P^t_{A}-\pi \Vert_{T.V.} =  \max_{A \in \Omega} \left \{ \frac{1}{2}\sum_{B \in \Omega} \vert P^t_{A}(B) - \pi(B) \vert \right \}.
\end{equation}
A question which arises naturally is to determine the rate of convergence to stationarity of the random walk, which is quantified by the mixing time
\begin{equation}
t_{mix}(\varepsilon)= \min \big\{t \in \mathbb{N}:  d(t)\leq \varepsilon\big\}.
\end{equation}

Our main result is the following.
\begin{theorem}\label{main2}
Let $n\ge 4$ and $q\ge 2$. For the Diaconis--Gangolli walk on $n\times n$ contingency tables with entries in $\mathbb{Z}/ q \mathbb{Z}$, with row sums $(r_1,\ldots, r_n) \bmod q$ and column sums $(c_1,\ldots, c_n) \bmod q$, we have that
\begin{itemize}
\item[(a)] (Upper bound) If $t \ge \frac{n^2}{4\left (1-\cos \frac{2\pi}{q}\right )} \log n + \frac{n^2}{\left (1-\cos \frac{2\pi}{q}\right )}c  \log \log (16n) \sqrt{\log n}\log q$, then
$$d(t) \leq  q^{-c},$$
for all $c \ge \frac{640}{\log \log (16n)}$.

\item[(b)](Lower bound) If $t\le \frac{n^2}{4 \left( 1- \cos \frac{2 \pi}{q}\right) } \log n   -  \frac{n^2  }{4\left( 1- \cos \frac{2 \pi}{q}\right) }(c+12) \log q $, then 
$$d(t) \geq 1 - q^{-c},$$
for all $c \geq 0$.
\end{itemize}
\end{theorem}

Set $t_{n, q} = \frac{n^2}{4 \left( 1- \cos \frac{2 \pi}{q}\right) } \log n$ and $\Delta_{n, q}=\frac{n^2}{\left (1-\cos \frac{2\pi}{q}\right)}  \log \log (16n) \sqrt{\log n}\log q$. Theorem \ref{main2} says that over $\mathbb{Z}/ q \mathbb{Z}$, the random walk mixes at time $t_{n, q}$ with window of order at most $\Delta_{n, q}$, i.e.
\begin{equation*}
\lim_{c \rightarrow \infty} \lim_{n \rightarrow \infty} d \left(t_{n, q}   -  c\Delta_{n, q} \right)= 1 \quad\footnote{The lower bound of Theorem \ref{main2} is in fact stronger than this.}\end{equation*}
and 
\begin{equation*}
 \lim_{c \rightarrow \infty} \lim_{n \rightarrow \infty} d \left(t_{n, q} + c\Delta_{n, q}  \right)= 0.
\end{equation*}

If $\log q=o\left (\frac{\sqrt{\log n}}{\log \log n}\right )$ then $\Delta_{n, q}= o(t_{n, q})$. In other words, the random walk exhibits \textit{cutoff}, a sharp transition from $d(t)\approx 1$ to $d(t)\approx 0$. 

If $n=2$, then the lazy random walk on contingency tables over $\mathbb{Z}/ q \mathbb{Z}$ is the same as the lazy random walk on $\mathbb{Z}/ q \mathbb{Z}$, which is known to not have cutoff \cite{Threads}. 

\subsection{Literature}
Contingency tables are used in statistics, in order to
display the results of tests and surveys. Diaconis
and Efron (\cite{DiaEfron}, \cite{DE}) developed the conditional volume test, which is a method for performing tests of great importance in such tables. The Diaconis-Efron test provides strong motivation for
sampling a contingency table with given row and column
sums uniformly at random. 

%The random walk that we study is a simpler version of the actual Diaconis and Gangolli
%random walk. The configuration space of the original walk is $n \times n$ matrices with specified row and column sums. The random walk suggests to pick two distinct rows and columns We flip a fair coin. If heads we add $T $ at the corners of the equilateral that we picked and if tails we add $-T$ to the corners. If the move that we attempt to make results in an entry being $-1$ or $2$, then we censor it. Similarly, we can define a process for contingency tables with non-negative, integer entries. 

Diaconis, Gangolli \cite{DiaGan} and Diaconis and Saloff-Coste (see page 373 of \cite{DiaStu}) were the first ones to introduce Markov chains for sampling contingency tables, although it is mentioned in \cite{DiaGan} that this chain has been used by practitioners.  Diaconis and Saloff-Coste
proved that if $N= \sum  c_i= \sum r_i$ and the number of rows and columns is fixed then the mixing time is of order $N^2$. Hernek \cite{Hernek} considered the case when the table has
two rows and proved that the same chain mixes in time polynomial in the number
of columns and $N$. Chung, Graham, and Yau \cite{CGY} proved that a 
modified version of the Diaconis and Saloff-Coste chain converges in time polynomial
in $N$, the number of rows, and the number of columns, given that $N$ is large.

Dyer, Kannan, and Mount \cite{DKM} found a new way to sample contingency tables using polytopes which provided the first truly polynomial-time algorithm (polynomial in the number of rows, the
number of columns, and the logarithm of $N$). Later, Morris \cite{Mor} refined their answers.

Dyer and Greenhill \cite{DG} applied coupling to get pre-cutoff at $n^2 \log N$ for the case of $2 \times 2$ heat-bath chain, a  different Markov chain on $2 \times n $ contingency tables. Matsui, Matsui and Ono \cite{MMO} extended the result of \cite{DG} for $2\times 2 \times \ldots \times 2 \times J$ contingency tables. Cryan, Dyer, Goldberg, Jerrum, Martin \cite{CDGJM} extended the result of \cite{DG} for the case where there is a constant number of rows.

A closely related problem to the mixing time is the enumerating problem. It is worth mentioning that there has been a rich literature in enumerating contingency tables (\cite{B1}, \cite{B2}, \cite{B3}, \cite{BH},\cite{B4}, \cite{BLSY}, \cite{BH2}, \cite{Be}, \cite{BDLV}, \cite{CDR} \cite{CRGM}, \cite{GM}, \cite{S}) and studying algorithms to sample contingency tables or approximate their number ( \cite{BliDia}, \cite{ChDiSu}, \cite{ChDiaHolLiu}, \cite{BSSV}, \cite{BBV}, \cite{BBV2}, \cite{BSV}, \cite{CD}, \cite{DiaStu}).  On a different note, Blanchet and Stauffer \cite{BS} provide a necessary
and sufficient condition so that the configuration model outputs a binary
contingency table with probability bounded away from $0$ and $N \rightarrow \infty$. 

Enumerating the graphs with a given degree sequence also has a vast literature and important applications. See for example, Bender and Canfield \cite{bender1978asymptotic}, Bollob\'{a}s \cite{bollobas1980probabilistic} and Wormald \cite{wormald1999models}. McKay and Wormald used the method of switchings to obtain in \cite{mckay1991asymptotic} an asymptotic formula for the number of labeled graphs with a given degree sequence in the case where the average degree is $o(\sqrt{n})$. See also \cite{mckay1997degree}, \cite{mckay1990asymptotic}, \cite{barvinok2013number} for other ranges of the degrees. In a recent breakthrough, Liebenau and Wormald \cite{liebenau2017asymptotic} obtained, among other things, the asymptotic number of $d$-regular graphs on $n$ vertices for all $d$. 

 \subsection{There is no Markovian coupling that could give Theorem \ref{main2}}
Coupling is a powerful technique to achieve upper bounds on the mixing time. This section is dedicated to proving that there is no Markovian coupling that would give optimal mixing time bounds. 
\begin{proposition}\label{prop:Markov}
Let $q\ge 2$ and let $T$ be a Markovian coupling time for the Diaconis--Gangolli walk in Theorem \ref{main2}, then
$$\pr{T > \frac{n^3}{100} } \geq \frac{1}{2}.$$
\end{proposition}
\begin{proof}
Each move of the walk is performed with probability $\frac{1}{2{n \choose 2}^2} $ for $q\ge 3$ and $\frac{1}{{n \choose 2}^2} $ for $q=2$. At each step we change four entries of the matrix. If there are two copies of the Markov chain $(C^1_t),(C^2_t),$ then we claim that
\begin{equation}\label{prob}
\pr{C^1_t=C^2_t \vert C^1_{t-1} \neq C^2_{t-1} } \leq \frac{6n}{{n \choose 2}^2},
\end{equation}
for every $t$. Notice that if $C^1_{t-1} \neq C^2_{t-1}$ then because of the degree restrictions, $C^1_{t-1}$ and $C^2_{t-1}$ have to differ in at least 4 coordinates.

If $C^1_{t-1} \neq C^2_{t-1}$ and $C^1_{t-1}, C^2_{t-1}$ differ in at least 9 entries, then \eqref{prob} holds because the left hand side of \eqref{prob} is zero. If $C^1_{t-1}, C^2_{t-1}$ differ in at least 5 vertices, the above probability is at most $\frac{6}{{n \choose 2}^2} $. If they differ in four entries $a,b,c,d$ (which is the minimum number of entries they can differ by), the only way to resolve these differences, without creating new differences, is if we change $a,b$ in $C^1_{t-1},$ and  $c,d$ in $C^2_{t-1}$ (or any other permutation of these letters). The other side of the box should be the same on the two chains so that we don't create new mismatched coordinates. The total probability of doing such a move is at most $\frac{{4\choose 2}n}{{n \choose 2}^2}$, which gives the right hand side term.

We will couple $T$ with a geometric random variable $R$ with probability of success  $p=6n/{n \choose 2}^2$ so that $T\ge R$ always. We have
$$\pr{R \ge \frac{1}{3p} }= \sum_{k\ge \frac{1}{3p}} (1-p)^{k-1} p \ge (1-p)^{1/(3p)} \geq \frac{1}{2}.$$
Thus, $\pr{T \ge \frac{1}{3p}  } \geq \pr{R\ge \frac{1}{3p} }  \geq \frac{1}{2}$ completing the proof of the statement.
\end{proof}
If $T$ is a coupling time then 
$$d(t) \leq \pr{T>t},$$
and, therefore Proposition \ref{prop:Markov} says that we cannot hope for a Markovian coupling that will give the upper bound of Theorem \ref{main2} even for $q=2$.

\section{The contingency table walk over $\mathbb{Z}/ q \mathbb{Z}$ as a random walk on a group}\label{group}
In this section, we explain how linear algebra and representation theory can be used to prove Theorem \ref{main2}.

Let $A_{i,j,k,l}$ denote the $n\times n$ matrix that has ones on the $(i,k), (j,l)$ positions, $-1$ on the $(i,l),(j,k)$ positions and zeros else where. If at time $t$ the Markov chain is at a contingency table $B_t$, then we choose matrices of the form $A_{i,j,k,l}$ and we add them to $B_t$. Let $A_{t+1}$ be the matrix we choose to add to $B_t$ at time $t+1$. Then, we have that
$$B_{t+1}= A_{t+1}+B_t=  B_0 + A_1 \ldots + A_{t+1} $$
So, instead of studying the Markov chain $(B_t)$, we can equivalently study the process $C_t$, where
$$C_t:=A_1 +\ldots + A_{t+1}. $$
The advantage of studying $C_t$ is that $C_t $ is a random walk on the group $G$, which consists of $n\times n$ contingency tables with entries over $\mathbb{Z}/ q \mathbb{Z}$ and row sums and column sums zero. This is summarized more formally in the following lemma.
\begin{lemma}\label{lm:translate}
For the processes $(B_t)$ and $(C_t)$, we have that
$$\Vert \mathbb{P}_{x}(B_t \in \cdot ) - \pi \Vert_{T.V.} = \Vert \mathbb{P}_{0}(C_t \in \cdot ) - U \Vert_{T.V.},$$
where $x \in \Omega$, $\pi $ is the uniform measure on $\Omega$ and $U$ is the uniform measure on $G$.
\end{lemma}
From now on, we will study the random walk $C_t $ on $G$. The next lemma characterizes $G$.
\begin{lemma}\label{isom}
The group $G$ of $n\times n$ contingency tables with entries over $\mathbb{Z}/ q \mathbb{Z}$ and row sums and column sums zero, satisfies $G \sim (\mathbb{Z}/ q \mathbb{Z})^{(n-1)^2}.$

\end{lemma}
\begin{proof}
$G$ is a vector subspace of $(\mathbb{Z}/ q \mathbb{Z})^{n \times n}$ over $\mathbb{Z}/ q \mathbb{Z}$. In fact, $G \sim (\mathbb{Z}/ q \mathbb{Z})^{n \times n}/ A $ where $A$ is the subgroup of $(\mathbb{Z}/ q \mathbb{Z})^{n \times n}$ generated by $2n-1$ relations, that are setting $2n-1$ rows and columns to be equal to zero. Therefore, $\dim_{\mathbb{Z}/ q \mathbb{Z}} G= (n-1)^2$ and this finishes the proof.
\end{proof}

Let $B_{i,j}$ be the matrix that has ones in positions $(i,j),  (i+1,j)$, $-1$ in positions $(i, j+1),(i+1,j+1)$ and zero everywhere else. To diagonalize the matrix $P$ we will need the fact that $\{B_{i,j}\}_{i,j=1}^{{n-1}}$ is a basis for $G$. We can see that they are a basis, because they are linearly independent and there are $(n-1)^2$ of them.  

\begin{definition} \label{til}
	Let $\tilde{ A}_{i,j,k,l} \in (\mathbb{Z}/ q \mathbb{Z})^{(n-1) \times (n-1)}$ be the matrix that has ones on all positions $(a,c)$ that satisfy $i \leq a \leq j$ and $k \leq c \leq l$ and everywhere else zero. 
\end{definition}
 
Since 
	$$A_{i,j,k,l}= \sum_{a=i}^{j-1} \sum_{c=k}^{l-1} B_{a,c},$$
the matrix $\tilde A_{i,j-1,k,l-1}$ is the coordinates of $A_{i,j,k,l}$ with respect to the basis $(B_{i,j})$ of $G$.  
Similarly, we associate each element of $G$ with its coordinates with respect to the basis $(B_{ij})$. 

\subsection{Fourier Transform and the $\ell^2$ bound}
Let $X$ be a finite group and let $S$ be a symmetric set of generators. 
\begin{definition}
Let $P$ be the uniform measure on $S$. Let $\rho $ be a representation of $X$. Define the Fourier transform of $\rho$ with respect to $P$ to be
\begin{equation*}
\hat{P}(\rho)= \sum_{s \in S} P(s) \rho(s).
\end{equation*}
\end{definition}
Theorem 6 of \cite[Chapter 3E]{PD} says that the Fourier transforms of the irreducible representations of $G$ with respect to $P$ give all of the eigenvalues of $P,$ each one appearing with multiplicity being the dimension of the corresponding representation. The following lemma explains how to use the irreducible representations of $X$ to bound the mixing time of the Markov chain generated by $P$. It was first used in \cite{DS} and the rigorous proof can be found in \cite[Chapter 3]{PD}.
\begin{lemma}[Upper bound lemma]\label{lm:upperboundlm}
For the random walk on $X$ generated by $P$, we have that
\begin{equation}\label{upb}
4\Vert P_{id}^{t} - \pi \Vert_{T.V.}^2 \leq \sum^* d_{\rho} (\hat{P}(\rho)\hat{P}^*(\rho))^t,
\end{equation}
where $d_{\rho}$ is the dimension of a representation $\rho$ and the sum is over all irreducible representations $\rho$ of $X$, but the trivial one.
\end{lemma}

Apply these results to $X = G$ and $S$ being the set of all $\pm A_{i, j, k, l}$; or equivalently, $X = (\mathbb{Z}/ q \mathbb{Z})^{(n-1)^2}$ and $S$ being the set of all $\pm \tilde A_{i,j,k,l}$. 
Let $y,g \in (\mathbb{Z}/ q \mathbb{Z})^{(n-1)^2}$, define
$$\rho_y(g)= e^{\frac{2 \pi i<g,y>}{q}},$$
where $<g,y>= \sum_{i=1}^{(n-1)^2} g_iy_i$ is the inner product of $y,g$. Each $\rho_y$ is one dimensional, therefore it is irreducible. As explained in Lemma 2 of \cite{PD}, we have that
$$\sum d_i^2=q^{(n-1)^2},$$
where the sum is taken over all irreducible representations and $d_i$ is the dimension of each irreducible representation.
Therefore, the set $\{ \rho_y, y \in (\mathbb{Z}/ q \mathbb{Z})^{(n-1)^2}\}$ consists of all irreducible representations of $G$. 
 
The following proposition computes the Fourier transform of each irreducible $
\rho_y$ with respect to $P$.
\begin{lemma}\label{lm:eigen}
Let $y \in (\mathbb{Z}/ q \mathbb{Z})^{(n-1)^2}$, then
$$\hat{P}(\rho_y)= \frac{1}{{n \choose 2}^2} \sum_{i,j,k,l} \cos \frac{2 \pi <y, \tilde{A}_{i,j,k,l}>}{q}$$
where $\tilde{A}_{i,j,k,l}$ is defined in Definition \ref{til}.
\end{lemma}
\begin{proof}
The proof of the lemma follows from the definition of $\rho_y$ and the fact that 
 $$\exp\left (\frac{2 \pi i<\tilde{A}_{i,j,k,l},y>}{q}\right )+ \exp\left (\frac{-2 \pi i<\tilde{A}_{i,j,k,l},y>}{q}\right )= 2 \cos \frac{2 \pi <y, \tilde{A}_{i,j,k,l}>}{q}. $$
\end{proof}

Applying Lemmas \ref{lm:upperboundlm} and \ref{lm:eigen}, we obtain the following bound for the random walk $(C_t)$ on $G$  and the uniform measure $U$ on $G$ as in Lemma \ref{lm:translate}.
\begin{equation}\label{eq:distance:2}
\Vert P_{0} (C_t\in \cdot)- U \Vert_{T.V.}^2 \leq \sum_{y\in G\setminus \{0\}} \hat{P}(\rho_y)^{2t} =  \sum_{y\in G\setminus \{0\}} \left (\frac{1}{{n \choose 2}^2} \sum_{i,j,k,l} \cos \frac{2 \pi <y, \tilde{A}_{i,j,k,l}>}{q}  \right )^{2t}.
\end{equation}

\subsection{Bounding negative eigenvalues}\label{sec:negativeEigenvalue}
In this section, we show that the negative eigenvalues are bounded away from minus one and, therefore, the don't contribute much to the right hand side of \eqref{upb}.

Let $S=\{\pm A_{i, j, k, l}: 1\le i<j\le n, 1\le k<l\le n\}$ denote the set of generators of the Diaconis--Gangolli random walk of interest over $\mathbb{Z}/ q \mathbb{Z}$. Let $P$ be the transition matrix and $Q= \frac{1}{2}(I+P)$ where $I$ is the identity matrix. Notice that since all the eigenvalues $\hat{P}(\rho_y)$ of $P$ lie in $[-1,1]$, we have that all the eigenvalues of $Q$ are non-negative, real numbers.
 
Let $f: G \rightarrow \mathbb{R} $ be a function and let 
\begin{equation}
\mathcal{F}(f,f)= \sum_{x,y 
\in G}(f(x)+f(y))^2P(x,y) \mbox{ and } \tilde{\mathcal{F}}(f,f)= \sum_{x,y 
\in G}(f(x)+f(y))^2Q(x,y).
\end{equation}

We are going to use Lemma 4 of \cite{DS} to bound the negative eigenvalues of $P$ from below. For completeness, we rewrite the statement of Lemma 4 of \cite{DS} for our case.
\begin{lemma}[\cite{DS}, Lemma 4]\label{dirichlet}
If $\tilde{\mathcal{F}}(f,f) \leq A_* \mathcal{F}(f,f)$ for every $f: G \rightarrow \mathbb{R} $, then 
$$\hat{P}(\rho_y) \geq -1+ \frac{1}{A_*},$$
for every $y \in (\mathbb{Z}/q \mathbb{Z})^{(n-1)^2 }$.
\end{lemma}
We are going to use Lemma \ref{dirichlet} to prove the following bound on the eigenvalues of $P$.
\begin{lemma} \label{lm:negative}
For the Diaconis--Gangolli random walk on contingency tables over $\mathbb{Z}/q \mathbb{Z}$, we have that 
$$\hat{P}(\rho_y) \geq -\frac{28}{29},$$
for every $y \in (\mathbb{Z}/q \mathbb{Z})^{(n-1)^2 }$.
\end{lemma}

\begin{proof}
The proof is based on the technique of ``flows'' as presented in Theorem 2.3 of \cite{DS2}. For completion, we explain how this method works.

Notice that for every $f: G \rightarrow \mathbb{R} $, we have that 
\begin{align}\label{lazy}
\tilde{\mathcal{F}}(f,f) &=2\sum_{x 
\in G}f(x)^2  + \frac{1}{2}\sum_{x\neq y 
\in G}(f(x)+f(y))^2P(x,y)\cr
&= 2 \sum_{x 
\in G}f(x)^2  + \frac{1}{2}\mathcal{F}(f,f).
\end{align}

To bound the term $\sum_{x\in G}f(x)^2$, we observe that for $n \geq 3$ and for any $k<b<l$ and $i< j$, 
\begin{equation}\label{loop}
A_{i,j,k,l}=A_{i,j,k,b} + A_{i,j,b,l}, 
\end{equation}
which can be illustrated for the case $i=1=k,$ $j=2=b$ and $l=3$ as 
$$\left[ \begin{matrix}
 & 1  &0 & -1 &\\
 & -1 & 0  &1& \\
 & & \ldots & &
\end{matrix}  \right] = \left[ \begin{matrix}
 & 1   & -1 &0 &\\
 & -1   &1 &0& \\
 & & \ldots & &
\end{matrix}  \right] + \left[ \begin{matrix}
&0  & 1  & -1 &\\
&0   & -1  &1& \\
 & & \ldots & &
\end{matrix}  \right].$$
Let $x \in G$ and let $f: G \rightarrow \mathbb{R} $. Using equation \eqref{loop}, we have
$$2f(x) = \left [f(x) + f(x+ A_{i,j,k,b})\right ]- \left [f(x+ A_{i,j,k,b}) + f(x+ A_{i,j,k,b}+A_{i,j,b,l} ) \right ] + \left [ f(x+ A_{i,j,k,l} )+ f(x) \right ].$$
 
Applying Cauchy-Schwartz gives
$$\frac{4}{3}f(x)^{2} \le \left [f(x) + f(x+ A_{i,j,k,b})\right ]^{2}+\left [f(x+ A_{i,j,k,b}) + f(x+ A_{i,j,k,b}+A_{i,j,b,l} ) \right ] ^{2}+\left [ f(x+ A_{i,j,k,l} )+ f(x) \right ]^{2}.$$
Thus, by averaging, 
\begin{align*}
f(x)^2 \leq  \frac{3}{4 {n \choose 2} {n \choose 3}} \sum_{i< j; k<b<l} \bigg\{ &\left [f(x) + f(x+ A_{i,j,k,b})\right ]^{2}+\left [f(x+ A_{i,j,k,b}) + f(x+ A_{i,j,k,b}+A_{i,j,b,l} ) \right ] ^{2} \\
&+\left [ f(x+ A_{i,j,k,l} )+ f(x) \right ]^{2} \bigg\}.
\end{align*}

Using the identities $P(x, x+g) = \frac{1}{|S|}\textbf{1}_{g\in S}$, $|S| = {n \choose 2}^{2}$ if $q=2$ and $|S| = 2{n \choose 2}^{2}$ if $q>2$, we get 
\begin{align}
\sum_{x \in G} f(x)^2 &\leq  \frac{3}{4 {n \choose 2} {n \choose 3}}  \sum_{x \in G} \sum_{i< j; k<b<l}\bigg\{ \left [f(x) + f(x+ A_{i,j,k,b})\right ]^{2} \cr
&\qquad\qquad +\left [f(x+ A_{i,j,k,b}) + f(x+ A_{i,j,k,b}+A_{i,j,b,l} ) \right ] ^{2}+\left [ f(x+ A_{i,j,k,l} )+ f(x) \right ]^{2} \bigg\}  \cr
& \le \frac{9n}{4 {n \choose 2} {n \choose 3}}  \sum_{x \in G, g \in S}  ( f(x) + f(x+g ))^2 \le \frac{18n{n \choose 2}}{4  {n \choose 3}}  \sum_{x \in G, g \in S}  ( f(x) + f(x+g ))^2 P(x,x+g)\cr
&\le \label{seven} 14  \sum_{x \in G, g \in S}  ( f(x) + f(x+g ))^2 P(x,x+g).
\end{align}
 
Equation \eqref{lazy} and \eqref{seven} give that 
$$\tilde{\mathcal{F}}(f,f) \leq 28 \mathcal{F}(f,f) + \frac{1}{2}\mathcal{F}(f,f) \leq 29 \mathcal{F}(f,f)$$
and then, Lemma \ref{dirichlet} for $A_*=29$ gives that
$$\hat{P}(\rho_y) \geq - \frac{28}{29},$$
for every $y \in (\mathbb{Z}/q \mathbb{Z})^{(n-1)^2 }$.
\end{proof}

\section{Proof of Theorem \ref{main2}}
\subsection{Proof of the lower bound}
For the proof of the lower bound, we will use Wilson's lemma.
\begin{lemma}[Lemma 5, \cite{Wilson}]\label{W}
Let $\ep, R$ be positive numbers and $0<\gamma< 2-\sqrt{2} $. Let $F: X\to \R$ be a function on the state space $X$ of a Markov chain $(C_t)$ such that 
$$\expect{F(C_{t+1})\vert C_t) }= (1 - \gamma )F(C_t), \quad \expect{\left [F(C_t)- F(C_{t-1})\right ]^2 \vert C_t} \leq  R,$$ and 
$$t \leq \frac{ \log \max_{x\in X}F(x) + \frac{1}{2} \log( \gamma  \varepsilon/(4R))}{-\log (1 - \gamma )}.
$$ Then the total variation distance from stationarity at time $t$ is at least $1-\ep$.
\end{lemma}

\begin{proof}[Proof of the lower bound of Theorem \ref{main2}]
As in Definition \ref{til} and the discussion that follows, we represent each element of $G$ as an $(n-1)\times (n-1)$ matrix which is its coordinates with respect to the basis $(B_{ij})$ of $G$. 

Let $D_{a, b}$ be the $(n-1) \times (n-1)$ matrix with entries at $(2a-1,2b-1), (2a,2b)$ equal to 1 and at $(2a,2b-1), (2a-1,2b)$ equal to $-1$, while all the other entries are zero. Theorem 6 of Chapter 3E of \cite{PD} says that the functions $G_{a,b}(x)= \cos \left(\frac{2 \pi <x, D_{a,b}>}{q} \right) $ are eigenfunctions of the transition matrix $P$. To see this, consider 
$$D_{1,1}= \left( \begin{matrix}
&1 &-1 & 0 & \ldots \\
&-1 & 1 & 0 & \ldots \\
&0 &0 & 0 & \ldots \\
& &\ldots &  & \ldots 
\end{matrix} \right)$$
We have that $$G_{1,1}(C)= \cos \left(\frac{2 \pi (C(1,1)- C(1,2)- C(2,1)+ C(2,2))}{q} \right).$$
We can only change the value of $G_{1,1}(C_t)$ if we do a move that affects exactly one of the following entries $C_t(1,1)$, $C_t(1,2)$, $C_t(2,1) $ and $C_t(2,2)$. For example, the moves that change only the entry $C_t(2,2)$ correspond to $\tilde{A}_{i,j,k,l}$, as defined in Definition \ref{til}, for which the $(2,2)$ entry is the left corner of the rectangle formed by ones. There are $(n-2)^2$ such $\tilde{A}_{i,j,k,l}$. Using this observation, direct calculation gives
\begin{align*}
\expect{G_{1,1}(C_t) \vert C_{t-1}} &   = \left( 1-  \frac{4}{n^2} +  \frac{4}{n^2} \cos \left( \frac{2 \pi}{q}\right)\right) G_{1,1}(C_{t-1}) .
\end{align*}
Similarly, it holds for any $G_{a,b}$ that
\begin{align}\label{e2}
\expect{G_{a,b}(C_t) \vert C_{t-1}} 
& = \left( 1-  \frac{4}{n^2} +  \frac{4}{n^2} \cos \left( \frac{2 \pi}{q}\right)\right) G_{a,b}(C_{t-1}) .
\end{align}
Let $$F(x)= \sum_{a,b=0}^{\lfloor \frac{n-1}{2}\rfloor}G_{a,b}.$$
Then, we have that
$\max F = F(0)= \lfloor \frac{n-1}{2}\rfloor^{2}$ and \eqref{e2} gives that
\begin{equation}\label{f2}
\expect{F(C_t) \vert C_{t-1}}= \left( 1-  \frac{4}{n^2} +  \frac{4}{n^2} \cos \left( \frac{2 \pi}{q}\right)\right) F(C_{t-1}).
\end{equation}
Finally, we have that
\begin{equation}\label{v2}
\expect{( F(C_t)- F(C_{t-1}))^2 \vert C_{t-1}} \leq 64.
\end{equation}
This is because every move $\tilde{A}_{i,j,k,l}$ that we might choose to make, affects the $G_{a,b}$ who share a unique one with  $\tilde{A}_{i,j,k,l}$. For example,
$$\tilde{A}= \left( \begin{matrix}
&1 &1 & 1 & 0 & \ldots \\
&1 &1 & 1 & 0 &\ldots \\
&1 &1 & 1 & 0 & \ldots \\
&0 &0 & 0 & 0 &\ldots \\
& &\ldots &  & & \ldots 
\end{matrix} \right)$$
fixes all $G_{a,b}$, but $G_{2,2}$. Therefore, every move that we make can affect at most four $G_{a,b}$. Thus, $|F(C_t)- F(C_{t-1})|\le 8$ surely and so 
$ \expect{( F(C_t)- F(C_{t-1}))^2 \vert C_{t-1}}\le  64$.

Using equations \eqref{e2} and \eqref{v2} for $n \geq 4$, we have that $\gamma= \frac{4}{n^2}\left( 1- \cos \frac{2 \pi}{q}\right) \leq \frac{1}{2} < 2-\sqrt{2}$ and $R= 64$. Since $\frac{1}{-\log(1-\gamma)}\ge \frac{1}{\gamma}-1$,  Lemma \ref{W} asserts that if 
$$t\le \left (\frac{1}{\gamma}-1\right )\log \left (F(0) \gamma^{1/2}\ep^{1/2}R^{-1/2}/2\right )$$
then
$$d(t) \geq 1 - \varepsilon.$$
Writing $\ep = q^{-c}$ and performing routine algebraic manipulations, we get that if 
$$t \le \frac{n^2}{4 \left( 1- \cos \frac{2 \pi}{q}\right) } \log n   -  \frac{n^2  \log q }{4\left( 1- \cos \frac{2 \pi}{q}\right) } (c+12),$$ 
then  $d(t) \geq 1 - q^{-c}$ as desired.
\end{proof}
\subsection{Proof of the upper bound of Theorem \ref{main2}}
	By \eqref{eq:distance:1} and Lemma \ref{lm:translate}, we get
\begin{equation}\label{key}
d(t)^{2} \le \Vert \mathbb{P}_{0}(C_t \in \cdot ) - U \Vert_{T.V.}^{2}    \nonumber
\end{equation}
where $C_t$ and $U$ are as in Lemma \ref{lm:translate}. Combining this with \eqref{eq:distance:2}, we get
\begin{equation}\label{eq:distance:3}
d(t)^{2} \le \sum_{y} \hat{P}(\rho_y)^{2t}
\end{equation}
where the sum runs over all nonzero $(n-1)\times (n-1)$ matrices $y$ and 
$$\hat{P}(\rho_y)  =  \frac{1}{{n \choose 2}^2} \sum_{i,j,k,l} \cos \frac{2 \pi <y, \tilde{A}_{i,j,k,l}>}{q}.$$
For each $a\in \Z/q\Z$, let $N_a(y)$ be the number of quadruples $(i, j, k, l)$ with $1\le i\le j\le n-1$ and $1\le k\le l\le n-1$ such that $<y, \tilde{A}_{i,j,k,l}> = a$. Then, 
\begin{equation}\label{eq:rep1}
\hat{P}(\rho_y)  = \frac{1}{{n \choose 2}^2} \sum_{a \in \Z/q\Z} N_a(y)\cos \frac{2 \pi a }{q}.
\end{equation}
Note that $\sum _{a\neq 0} N_a(y) = {n \choose 2}^2 - N_0(y)$.

We need to show that for $t=\frac{n^2}{4\left (1-\cos \frac{2\pi}{q}\right )} \log n + \frac{n^2}{\left (1-\cos \frac{2\pi}{q}\right )}c \log\log(16n) \sqrt{\log n}\log q$, 
$$d(t)^{2}\le q^{-2c}.$$
To do so, we decompose the right-hand side of \eqref{eq:distance:3} into the sum over all $y$ with $\hat{P}(\rho_y)< 0$ and the sum of the rest, and denote the corresponding sums by $d_1$ and $d_2$.

By Lemma \ref{lm:negative}, each of the $\hat{P}(\rho_y)$ with $\hat{P}(\rho_y)< 0$ satisfies $\hat{P}(\rho_y)\ge -\frac{28}{29}$. Thus, 
\begin{equation}\label{eq:d1}
d_1\le q^{(n-1)^{2}}\left (\frac{28}{29}\right )^{2t} \le q^{n^{2}} q^{-8n^{2}-2c}\le q^{-2c}/2.
\end{equation}
where we used the assumption that $c\ge \frac{640}{\log \log(16n)}$ and $n\ge 3$ to get the estimate
$$t\ge \frac{n^2}{\left (1-\cos \frac{2\pi}{q}\right )}c \log\log(16n) \sqrt{\log n}\log q \ge 2c\log q+ 120 n^{2}\log q.$$

For the rest of the proof, we will show that $d_2\le q^{-2c}/2$. Assume that it holds, we have $d(t)^{2}\le d_1+d_2\le q^{-2c}$ as desired.

For each $(n-1)\times (n-1)$ matrix $y$ and each $1\le i\le j\le n-1$ and $1\le k\le l \le n-1$, we consider the minor $y_{i, j, k, l}$ obtained by restricting to the rows from $i$ to $j$ and the columns from $k$ to $l$, inclusively. We say that $[i, j]\times [k, l]$ is a nonzero box if the sum of the entries in this minor is nonzero. Denote by $N(y)$ the number of non-zero boxes of $y$. Observe that $\sum _{a\neq 0} N_a(y) = N(y)$. When $\hat{P}(\rho_y)\ge 0$, we have from \eqref{eq:rep1} that
\begin{equation}\label{key}
0\le \hat{P}(\rho_y)  \le \frac{1}{{n \choose 2}^2} \left [N_0(y)+ \left (\sum_{a \neq 0} N_a(y)\right )\cos \frac{2 \pi }{q}\right ] = 1 - \frac{\left (1 - \cos \frac{2 \pi }{q}\right ) N(y)}{{n \choose 2}^2} .\nonumber
\end{equation}

Thus, 
$$d_2\le \sum_{y: \hat{P}(\rho_y)\ge 0} \left (1 - \frac{\left (1 - \cos \frac{2 \pi }{q}\right ) N(y)}{{n \choose 2}^2} \right )^{2t}\le \Sigma$$
where 
\begin{equation}
\Sigma:= \sum \left( 1- \frac{(1-\cos \frac{2\pi}{q} )N(y)}{{n \choose 2}^2} \right)^{2t}. \label{eq:sum:positive}
\end{equation}
where the sum is taken over all matrices $y\neq 0$ with $N(y) \le \frac{{n \choose 2}^2}{1-\cos \frac{2\pi}{q} }$.

It remains to show that 
\begin{equation}\label{eq:final}
\Sigma\le q^{-2c}/2 
\end{equation}
and for that we need to control the number of matrices $y$ with a prescribed range of $N(y)$.

The number $N(y)$ of nonzero boxes of a matrix $y$ can vary significantly just by changing an entry of $y$. For example, if $y$ is the zero matrix then $N(y) = 0$. If we add a single entry $1$ at around the middle of the matrix $y$, then $N(y) \approx (n/2)^{4}$ which is significantly larger. Now, if the single entry $1$ is at position $(1, 1)$ (for example) instead of at the middle of the matrix, then $N(y)\approx n^{2}$. 

It turns out to be useful to look at one dimensional version of the nonzero boxes. Let $u$ be a vector in $(\mathbb Z/q\mathbb Z)^{n-1}$. For any $1\le i\le j\le n-1$, the interval $[i, j]$ is said to be a nonzero interval of $u$ if the sum of the entries of the vector $[u(i), u(i+1), \dots, u(j)]$ is nonzero in $\mathbb Z/q \mathbb Z$. Let $S(u)$ be the number of nonzero intervals in $u$.

In order to control $S(u)$, we introduce the following definition.
\begin{definition}\label{skeleton}
Let the \textbf{{\it skeleton}} of $u$ to be the set $I(u) = \{i_1, i_2, \dots, i_{s(u)}\}\subset [1, n-1]$ with
\begin{itemize}
\item $i_1\ge 1$ being the smallest index such that $u_{i_1}\neq 0$,
\item $i_k$ being the smallest index such that $i_k\ge i_{k-1}+2$ and $u_{i_k}\neq 0$, for all $2\le k\le s(u)$,
\item  $u_i=0$ for all $i\ge i_{s(u)}+2$.
\end{itemize}
\end{definition}

For example, if $u$ has nonzero entries at positions $3, 4, 5, 8, 9$ and the rest are 0 then $I(u) = \{3, 5, 8\}$ and $s(u)=3$.
\begin{equation} \begin{array}{ccccccccccccccccccccccccccccccccc} 
 	u & =&[& 0 & 0& *& *& *& 0& 0& *& *& 0& 0 &\dots & 0]  \\	
 	 &  & &  &  & \downarrow&  & \downarrow&  & & \downarrow\\
 	I(u) &=& \{ &  &  & 3&  & 5&  & & 8 & & &  & & \} 
 \end{array}\nonumber
 \end{equation}
 
 Observe that for any row vector $u$, the skeleton size $s(u)$ is at most $n/2$ and indices of the nonzero entries of $u$ form a subset of $I(u)\cup (I(u)+1)$. Thus, the number of nonzero elements in $u$ is at most $2s(u)$.

The number of nonzero intervals $S(u)$ is controlled by the skeleton size $s(u)$ as follows.
\begin{lemma}\label{lm:nonzerointerval:skeleton}
Let $u$ be a nonzero $1\times (n-1)$ vector in $(\Z/q\Z)^{n-1}$. Then for every $n\ge 3$ and $q\ge 2$, we have 
\begin{equation}
S(u) \ge s(u)\left (n-s(u)\right ).\nonumber
\end{equation}
Furthermore, the number of nonzero intervals with one or both endpoints belonging to \break $I(u)\cup \left  (I(u)\pm 1\right  )$ is at least $s(u)\left (n-s(u)\right )$.
\end{lemma}

\begin{proof}
Let the skeleton of $u$ be $I(u) = \{i_1, i_2, \dots, i_{s(u)}\}$ with $i_1<i_2<\dots<i_{s(u)}$. We first claim that the number of nonzero intervals $[i, j]$ with either $j\in \{i_1-1, i_1\}$ or $i\in \{i_1, i_1+1\}$ or both \footnote{This guarantees that one or both endpoints of $[i, j]$ belong to the set $\{i_1, i_1\pm 1\}$.} is at least $n-1$.
 
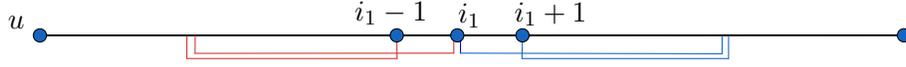
\definecolor{qqqqff}{rgb}{0,0,1}\definecolor{dtsfsf}{rgb}{0.8274509803921568,0.1843137254901961,0.1843137254901961}\definecolor{rvwvcq}{rgb}{0.08235294117647059,0.396078431372549,0.7529411764705882}

\begin{tikzpicture}[line cap=round,line join=round,>=triangle 45,x=1cm,y=1cm]\clip(-7.5,-1) rectangle (7.066666666666675,1);\draw [line width=0.7pt] (-4.748292682926833,0.008048780487805612)-- (6.737617560975608,0.009178311984559472);\draw (-0.7,0.5933333333333319) node[anchor=north west] {$i_1-1$};\draw (0.6766352201257884,0.5666666666666652) node[anchor=north west] {$i_1$};\draw (1.4235220125786194,0.58) node[anchor=north west] {$i_1+1$};

\draw [line width=0.4pt,color=dtsfsf] (-0.005009054254589529,0.008515237925035411)-- (0,-0.3);
\draw [line width=0.4pt,color=dtsfsf] (0,-0.3)-- (-2.79545892682927,-0.3);

\draw [line width=0.4pt,color=dtsfsf] (0.75325057258594,-0.22848190910686011)-- (-2.685602731707319,-0.23118259343744443);
\draw [line width=0.4pt,color=dtsfsf] (0.75325057258594,0.008589805644358127)-- (0.75325057258594,-0.22854985498573813);

\draw [line width=0.4pt,color=dtsfsf] (-2.79545892682927,-0.3)-- (-2.79545892682927,-0.007025793998946408);
\draw [line width=0.4pt,color=dtsfsf] (-2.6844461300645013,0.008251740430589637)-- (-2.685602731707319,-0.23118259343744443);

\draw [line width=0.4pt,color=qqqqff] (0.8427981558269106,0.008598611808918098)-- (0.842001756097559,-0.22848190910686011);

\draw [line width=0.4pt,color=rvwvcq] (1.6669056405069105,0.008679655047189164)-- (1.6598200975609736,-0.3);
\draw [line width=0.4pt,color=rvwvcq] (1.6598200975609736,-0.3)-- (4.406224975609755,-0.3);
\draw [line width=0.4pt,color=rvwvcq] (4.406224975609755,-0.3)-- (4.406224975609755,-0.0043251096683620946);
\draw [line width=0.4pt,color=rvwvcq] (0.842001756097559,-0.22848190910686011)-- (4.320781268292681,-0.23658396209861304);
\draw [line width=0.4pt,color=rvwvcq] (4.320781268292681,-0.23658396209861304)-- (4.320781268292681,-0.0043251096683620946);

\draw (-5.310314465408809,0.406666666666665) node[anchor=north west] {$u$};
\begin{scriptsize}\draw [fill=rvwvcq] (-4.748292682926833,0.008048780487805612) circle (2.5pt);\draw [fill=rvwvcq] (6.737617560975608,0.009178311984559472) circle (2.5pt);\draw [fill=rvwvcq] (-0.005009054254589529,0.008515237925035411) circle (2.5pt);\draw [fill=rvwvcq] (1.6669056405069105,0.008679655047189164) circle (2.5pt);\draw [fill=rvwvcq] (0.7999991548211156,0.008594402928765016) circle (2.5pt);\end{scriptsize}\end{tikzpicture}

Indeed, let $1\le t\le n-1$ be any index. If $t< i_1$ then either $[t, i_1-1]$ or $[t, i_1]$ is a nonzero interval. If $t=i_1$ then $[t, i_1]$ is a nonzero interval. If $t>i_1$ then $[i_1, t]$ or $[ i_1+1, t]$ is a nonzero interval.

Applying the same argument for $i_2$ in place of $i_1$ and $t$ running from $1$ to $n-1$ except for $t=i_1, i_1+1$ (to avoid double counting), we obtain at least $n-3$ other nonzero intervals $[i, j]$ with either $j\in \{i_2-1, i_2\}$ or $i\in \{i_2, i_2+1\}$ or both. Keep running this argument for $i_3, \dots, i_{s(u)}$ gives
$$(n-1) + (n-3) + (n-5)+ \dots + (n-2s(u)+1) = ns(u)-s(u)^{2}$$
nonzero intervals as claimed.
\end{proof}

Using a similar argument, we can control the number nonzero boxes by the number of nonzero intervals. 
\begin{lemma}\label{lm:nonzerointerval:nonzeroboxes}
Assume that an $(n-1)\times (n-1)$ matrix $y$ has some nonzero rows $i_1, i_2, \dots, i_k$ with $|i_{m}-i_{l}|\ge 2$ for all $m\neq l$. Let $S_{m}$ be the number of nonzero intervals of row $i_m$. Then for every $n\ge 3$ and $q\ge 2$, we have
$$N(y)\ge   S_1(n-1)+S_2(n-3)+ \dots+ S_k(n-2k+1).$$
\end{lemma}
\begin{proof}
Consider a nonzero interval $[p, q]$ of row $i_m$ for some $1\le p\le q\le n-1$ and $1\le m\le k$. For each $1\le r \le i_m-1$, either the box $[r, i_m-1]\times [p, q]$ or the box $[r, i_m]\times [p, q]$ is a nonzero box. Similarly, for each $i_m+1 \le r\le n-1$, either the box $[i_{m}+1, r]\times [p, q]$ or the box $[i_m, r]\times [p, q]$ is a nonzero box. For $r = i_m$, the box $[i_{m}, r]\times [p, q]$ is itself a nonzero box. Thus, each nonzero interval $[p, q]$ contributes at least $n-1$ nonzero boxes that touch \footnote{We say that a box touches a row or a column if it has that row or column as a boundary. In other words, the box $[i, j]\times [k, l]$ touches rows $i, j$ and columns $k, l$.} the columns $p, q$ and either touch the row $i_m$ or the row $i_m-1$ from above or the row $i_m+1$ from below. Taking union over all nonzero intervals of row $i_m$, there are at least $S_{m}(n-1)$ nonzero boxes that either touch the row $i_m$ or the row $i_m-1$ from above or the row $i_m+1$ from below. Taking union over $m$ and subtracting the multiple-counted boxes, we conclude that the number of nonzero boxes is at least $S_1(n-1)+S_2(n-3)+ \dots+ S_k(n-2k+1)$.
\end{proof}

Now, we give an upper bound on the number of matrices $y$ with $o(n^{3})$ nonzero boxes.
\begin{lemma}\label{lm:nonzeroboxes:small}
Let $0<\ep_n\le \frac{1}{80}$ be any number. Let $w$ be a positive integer satisfying $1\le w\le \ep_n n$. Then for every $n\ge 3$ and $q\ge 2$, the number of $(n-1)\times (n-1)$ matrices $y$ in $\Z/q\Z$ with 
\begin{equation}
 N(y)\le \left (w+\frac{1}{2}\right )n^{2}\label{eq:somebound2}
\end{equation}
  is at most 
  $q^{32w}n^{2w+1+60\ep_n w}$.

  Moreover, if $ w\le \min\left \{\frac{1}{100\ep_n }, \ep_n n\right \}$, the number of such matrices $y$ is at most $q^{32w}n^{2w}$. Also, for such $w$, the number of matrices $y$ with 
  \begin{equation}
 N(y)\le \left (w-40 w \ep_n\right )n^{2}\label{eq:somebound2:2}
\end{equation}
is at most $q^{16 w}n^{2(w-1)}$. 
\end{lemma}

In the above statement, only the constant $2$ in $n^{2w}$ is important. All of the other constants $16, 32, 80, 100$ are merely for explicitness. Their exact values do not play any significant role. This also holds for other (big) constants in the rest of the proof.
\begin{proof}
To prove the first part of the lemma, let $y$ be an $(n-1)\times (n-1)$ matrix satisfying \eqref{eq:somebound2}. Let $i_1< i_2<\dots<i_p$ be some nonzero rows of $y$ with $i_m\ge i_{m-1}+2$ for each $m$. For each row $i$, let $S_i$ be the number of nonzero interval in that row and $0\le s_i\le \frac{n}{2}$ be the size of the its skeleton. By Lemma \ref{lm:nonzerointerval:skeleton}, $S_i \ge s_i(n-s_i)$. Note that $s_{i_j}\ge 1$ for all $j=1, \dots, p$ as they are nonzero rows.
By applying Lemma \ref{lm:nonzerointerval:nonzeroboxes} to the row $i_j$, we get
\begin{equation}
\left (w+\frac{1}{2}\right )n^{2} \ge  S_{i_j}(n-1)\nonumber
\end{equation}
which together with Lemma \ref{lm:nonzerointerval:skeleton} give
\begin{equation}
\left (w+\frac{1}{2}\right )n^{2} \ge  s_{i_j}(n-s_{i_j})(n-1).\nonumber
\end{equation}
Thus, for each $j=1, \dots, p$, 
\begin{equation}\label{eq:somebound3:1}
s_{i_j} \le 2s_{i_j}\left (1-\frac{s_{i_j}}{n}\right )\le 4w\le 4\ep_n n.
\end{equation}

By applying Lemma \ref{lm:nonzerointerval:nonzeroboxes} to the rows $i_1, \dots, i_p$, we get
\begin{equation}
\left (w+\frac{1}{2}\right )n^{2} \ge  S_{i_1}(n-1)+S_{i_2}(n-3) + \dots+S_{i_p}(n-2p+1)\label{eq:somebound3}
\end{equation}
and by Lemma \ref{lm:nonzerointerval:skeleton},
\begin{equation}
\left (w+\frac{1}{2}\right )n^{2} \ge  s_{i_1}(n-s_{i_1})(n-1)+s_{i_2}(n-s_{i_2})(n-3) + \dots+s_{i_p}(n-s_{i_p})(n-2p+1).\label{eq:secsomebound3}
\end{equation}

Combining \eqref{eq:somebound3:1} and \eqref{eq:secsomebound3}, we obtain 
\begin{equation}
 w+\frac{1}{2} \ge  \sum_{j=1}^{p} s_{i_j}\left (1-\frac{s_{i_j}}{n}\right )\left (1 - \frac{2j-1}{n}\right )\ge  \sum_{j=1}^{p} s_{i_j}\left (1-4\ep_n\right )\left (1 - \frac{2j-1}{n}\right ). \label{eq:somebound55} 
\end{equation}
Thus, 
$$w+\frac{1}{2} \ge (1-4\ep_n)\sum_{j=1}^{p} \left (1 - \frac{2j-1}{n}\right ) = (1-4\ep_n)p\left (1-\frac{p}{n}\right ) \ge  p/4 ,$$
because $p\le n/2$ by the assumption that $i_m\ge i_{m-1}+2$ for $m=2, \dots, p$. So, $p\le 4w+2\le 6\ep_n n$. Plugging this into \eqref{eq:somebound55}, we obtain
 \begin{equation}
\left (w+\frac{1}{2}\right )\left (1+20\ep_n\right ) \ge \left (w+\frac{1}{2}\right )\left (1-4\ep_n\right )^{-1} \left (1-12\ep_n\right )^{-1}  \ge \sum_{j=1}^{p} s_{i_j} \ge p.\label{eq:somebound5}
\end{equation}

We conclude that if there are $p$ nonzero rows whose indices are of distance at least 2 from each other, then $p\le \left (w+\frac{1}{2}\right )\left (1+ 20\ep_n\right )$. Let 
$$f(w) = \left \lfloor \left (w+\frac{1}{2}\right )\left (1+ 20\ep_n\right )\right \rfloor.$$  
 Note that $f(w)\le 2w\le n/2$. Thus, there are at most 
 $$\sum_{p=1}^{f(w)} {n\choose p}2^{p}\le f(w){n\choose f(w)}2^{f(w)}$$
  ways to choose the indices of the nonzero rows of $y$, and it's always true that there are at most $2f(w)$ nonzero rows. Similarly, there are at most that many ways to choose the indices of the nonzero columns of $y$. Let $\mathcal I$ be the set of the chosen columns. We have $|\mathcal I|\le 2f(w)$. Applying \eqref{eq:somebound5} to the sequence of chosen rows with odd indices $i_1, \dots, i_p$, we see that the number of ways to choose the corresponding value $s_{i_j}$ for these rows is at most the number of ways to choose a sequence of positive integers $s_{i_1}, \dots, s_{i_p}$ satisfying \eqref{eq:somebound5} for some value $p$. And that number is at most
\begin{equation}
\sum_{a = 1}^{f(w)} \sum_{p=1}^{a} \left | \left \{(s_{i_1}, \dots, s_{i_p}: s_{i_1}+ \dots+ s_{i_p} = a\right \}\right | = \sum_{a = 1}^{f(w)} \sum_{p=1}^{a} {{a-1}\choose {p-1}} = \sum_{a = 1}^{f(w)} 2^{a-1} \le 2^{f(w)}.\nonumber
\end{equation}

Having chosen the $s_i$ for the rows with odd indices, by the definition of skeletons and the fact that the nonzero entries lie on the columns in $\mathcal I$ and \eqref{eq:somebound5}, the number of ways to choose these rows in $(\mathbb Z/ q\mathbb Z)^{n-1}$ is at most
\begin{equation}
\prod_{j=1}^{p}q^{2s_i}{|\mathcal I| \choose s_i}  \le \prod_{j=1}^{p}q^{2s_i}|\mathcal I| ^{ s_i}\le q^{2f(w)}(2f(w))^{f(w)}.\label{eq:choose:smallrow}
\end{equation}
Similarly for the rows with even indices.

All in all, the number of choices for $y$ is at most
$$\left [f(w){n\choose f(w)}2^{f(w)}\right ]^{2} \left [2^{f(w)}\right ]^{2} \left [q^{2f(w)}(2f(w))^{f(w)} \right ]^{2}  \le q^{32w}n^{2f(w)} \le q^{32w}n^{2w + 1+ 60\ep_n n} .$$

Now, for the second part of the lemma, if $w\le \frac{1}{100\ep_n }$ then $f(w)\le w+\frac{5}{6}$. In that case, $f(w)\le w$ because both are integers. Thus, the number of matrices $y$ satisfying \eqref{eq:somebound2} is at most $q^{32w} n^{2w}$ as claimed. To bound the number of matrices $y$ that satisfy \eqref{eq:somebound2:2}, we use the same argument, with $w+1/2$ being replaced by $w-40 w \ep_n$ throughout. We conclude that the number of matrices $y$ satisfying \eqref{eq:somebound2:2} is at most $q^{16 w}n^{2f(w)}$ where $f(w)$ is the integer part of 
$$\left (w-40 w \ep_n\right )\left (1+ 20\ep_n\right ),$$
which is strictly less than $w$. Therefore, $f(w)\le w-1$. This completes the proof.
\end{proof}

To handle the matrices with a large number of nonzero boxes, we need a stronger version of Lemma \ref{lm:nonzerointerval:nonzeroboxes}.

\begin{lemma}\label{lm:nonzerointerval:nonzeroboxes:stronger} 
Let $y$ be an $(n-1)\times (n-1)$ matrix. Let $\Psi_i$ be a collection of nonzero intervals on each row $i$ such that for every two consecutive rows $i$ and $i+1$, we have that $\Psi_i \cap \Psi_{i+ 1}= \emptyset$. Then for every $n\ge 3$ and $q\ge 2$, it holds that
\begin{equation}
2N(y)\ge \left (|\Psi_1| + |\Psi_2| +\dots +|\Psi_{n-1}|\right )n. \label{eq:oddinterval:oddboxes:general}
\end{equation}
\end{lemma}

\begin{proof}
For each row $i$, let $\mathfrak P_i$ be the collection of nonzero boxes of the form $[i, j]\times [k, l]$ or $[i+1, j]\times [k, l]$ where $j\ge i+1$ and $[k, l]\in \Psi_i$ together with the nonzero boxes $[i, i]\times [k, l]$ (for all $[k, l]\in \Psi_i$). Note that for each $[k, l]\in \Psi_i$, since the interval $[k, l]$ is a nonzero interval in row $i$, either $[i, j]\times [k, l]$ or $[i+1, j]\times [k, l]$ must be a nonzero box. Thus, $|\mathfrak P_i|\ge |\{j: j\ge i+1\}||\Psi_i|+ |\Psi_i|=  (n-i)|\Psi_i|$. By the hypothesis, the $(\mathfrak P_i)_{i=1}^{n-1}$ are disjoint. Thus,
\begin{equation}\label{o}
N(y)\ge \sum_{i=1}^{n-1}|\mathfrak P_i| \ge \sum_{i=1}^{n-1} (n-i)|\Psi_i|.
\end{equation}

Likewise, let $\mathfrak Q_j$ be the collection of nonzero boxes of the form $[j, i]\times [k, l]$ or $[j, i-1]\times [k, l]$ where $j\le i-1$ and $[k, l]\in \Psi_i$ together with the nonzero boxes $[i, i]\times [k, l]$ (for all $[k, l]\in \Psi_i$). Then the $\mathfrak Q_i$ are disjoint and $|\mathfrak Q_i|\ge i|\Psi_1|$. Therefore, 
\begin{equation}\label{t}
N(y)\ge \sum_{i=1}^{n-1}|\mathfrak Q_i| \ge \sum_{i=1}^{n-1} i|\Psi_i|.
\end{equation}
Adding up \eqref{o} and \eqref{t}, we obtain the desired bound.
\end{proof}

Now, we bound the number of matrices $y$ with a large number of nonzero boxes.
\begin{lemma}\label{lm:nonzeroboxes:large}
Let $\ep_n$ be a positive number satisfying $\frac{1}{\sqrt {\log n}}\le \ep_n\le \frac{1}{80}$. Let $w\ge \ep_n n$ be an integer. For every $n\ge 3$ and $q\ge 2$, the number of $(n-1)\times (n-1)$ matrices $y$ with 
\begin{equation}
N(y)\le \left (w+\frac{1}{2}\right )n^{2}\label{eq:somebound2:3}
\end{equation}
  is at most 
  $n^{2w+C_1 \ep_n w}$ where $C_1 = \frac{200 \log q}{\ep_n \sqrt{\log n}} + 80$.
\end{lemma}
\begin{proof} We will treat the ``big" rows and ``small" rows separately. \newline  
\textbf{Step 1: Big rows.} A row $i$ is said to be a big row if either $S_{i-1}\ge \ep_n n(n-\ep_n n)$ or $S_{i}\ge \ep_n n(n-\ep_n n)$ or $S_{i+1}\ge \ep_n n(n-\ep_n n)$ where $S_j$ is the number of nonzero interval of row $j$. In other words, either row $i$ or $i\pm 1$ has a large number of nonzero intervals.

Let $\mathcal B$ be the set of indices $i$ of the big rows in $y$.  

Let $M$ be the number of rows $i$ with 
\begin{equation}\label{def:bigrow}
S_i\ge \ep_n n(n-\ep_n n)\ge \ep_n n^{2}/2.
\end{equation}
Assume that there is a sequence of rows $i_1 < i_2< \dots < i_p$ satisfying \eqref{def:bigrow} with $i_{m}\ge i_{m-1}+2$. By \eqref{eq:somebound3}, we get
$$\left (w+\frac{1}{2}\right )n^{2} \ge \frac{1}{2}\ep_n n^{2} \left [(n-1)+\dots(n-2p+1)\right ] = \frac{1}{2}\ep_n n^{2} p(n-p)\ge \ep_n n^{2} \frac{pn}{4}.$$
Thus, $p\le \frac{ 4\left (w+\frac{1}{2}\right )}{\ep_n n}\le 6w$. And so, 
\begin{equation}
M\le 2p\le \frac{12w}{\ep_n n}\label{eq:bigrow}.
\end{equation}

Since $|\mathcal B|\le 3M\le \frac{36w}{\ep_n n}$, the number of ways to choose the set $\mathcal B$ and the realizations of the rows with indices belonging to $\mathcal B$ is at most
\begin{equation}
 2^{n} (q^{n-1})^{36 w/(\ep_n n)}\le 2^{n} q^{36 w/\ep_n}\le n^{40  w\ep_n (\log q)/a^{2}} \le n^{40  w\ep_n (\log q)/a},\label{eq:choose:bigrow}
\end{equation}
where $a = \ep_n \sqrt{\log n}$ and in the last 2 inequalities, we used the assumption that $a\ge 1$.

\textbf{Step 2: Small rows.} It's left to control the set of small rows $ [1, n-1]\setminus \mathcal B=:\mathcal S$. We will use Lemma \ref{lm:nonzerointerval:nonzeroboxes:stronger} for which we need to define the $\Psi_i$ carefully.

\quad \textbf{Step 2.1: Definition of $\Psi_i$ and $r_i$.} 
 
For $i\in \mathcal B$, we set $\Psi_i = \emptyset$. For $i\in \mathcal S$, we consider the cases that $i$ is odd and even separately and define $\Psi_i$ (and another quantity denoted by $r_i$, which is in essence, the skeleton size of row $i$) differently for each case.

\quad \quad  \underline{\textit{Assume that $i\in \mathcal S$ is odd.}} Let $I_i$ be the skeleton of row $i$ and $r_i := |I_i|$. Note that by Lemma \ref{lm:nonzerointerval:skeleton}, the fact that $r_i\le n/2$ and the assumption that row $i\in \mathcal S$, $r_i\le \ep_n n$. For an odd index $i$, let $\Psi_i$ be the collection of nonzero intervals of row $i$ with one or both endpoints belonging to the set $\mathcal I_i:= I_i\cup (I_i\pm 1)$. By Lemma \ref{lm:nonzerointerval:skeleton}, 
\begin{equation}
|\Psi_i|\ge r_i(n-r_i)\ge r_i(n-\ep_n n).\label{eq:small:1}
\end{equation}

\quad \quad \underline{\textit{Assume that $i\in \mathcal S$ is even.}} Note that the $\Psi_{i\pm 1}$ are defined in Step 2.1.1. Let $\mathcal K_i = \mathcal I_{i-1}^{c} \cap \mathcal I_{i+1}^{c}$. Observe that any interval $[h, h']$ with $h, h'\in \mathcal K_i$, it holds that $[h, h']$ does not belong to $\Psi_{i-1} \cup \Psi_{i+1}$. Define $\Psi_i$ to be the collection of nonzero intervals with both endpoints in $\mathcal K_i$. This guarantees that the assumption of Lemma \ref{lm:nonzerointerval:nonzeroboxes:stronger} is satisfied.

To define $r_i$, let $\mathcal L_i = \mathcal J_{i-1}^{c} \cap \mathcal J_{i+1}^{c},$ where  $\mathcal J_j:= I_j\cup (I_j\pm 1) \cup (I_j\pm 2)$. We have $\mathcal L_i\subset \mathcal K_i$ and 
\begin{equation}
|\mathcal K_i|\ge |\mathcal L_i|\ge n-5|I_{i-1}|-5|I_{i+1}|\ge n-5r_{i-1}-5r_{i+1}\ge n-10\ep_n n.\label{eq:bound:L}
\end{equation}

We consider a version of skeleton $\tilde{\mathcal I_i}=\{i_1, i_2, \dots, i_s\}$ restricted to the set $\mathcal L_i$ as follows. Let $u=(y_{i, 1}, \dots, y_{i, n-1})\in (\Z/q\Z)^{n-1}$ be the vector of row $i$. Let $i_1$ be the smallest index in $\mathcal L_i$ such that $u_{i_1}\neq 0$. Let $i_k$ be the smallest index such that $i_k\in \mathcal L_i$, $i_k\ge i_{k-1}+2$ and $u_{i_k}\neq 0$ for all $2\le k\le s$. Here $s$ is the largest index for which this process has to stop, meaning, $u_{h}=0$ for all $h$ satisfying both $h\ge i_{s}+2$ and $h\in \mathcal L_i$. Set $r_i:=s$. Observe that $r_i\le \ep_n n$ because $i\in \mathcal S$ and $r_i$ is at most the size of the skeleton of row $i$.

By the same argument as in the proof of Lemma \ref{lm:nonzerointerval:skeleton},  the number of nonzero intervals with both endpoints belonging to the set $\{i_1, i_1\pm 1, \dots, i_{r_i}, i_{r_i}\pm 1\}$ or one endpoint in this set and the other in $\mathcal K_i$ is at least $r_i(|\mathcal K_i|-r_i)$. By the definition of $\mathcal L_i$, this set is a subset of $\mathcal K_i$, and so all such intervals belong to $\Psi_i$. Therefore,
\begin{equation}
|\Psi_i|\ge r_i(|\mathcal K_i|-r_i). \nonumber
\end{equation}
From that and \eqref{eq:bound:L}, we get
\begin{equation}
|\Psi_i|\ge  r_i(n-10\ep_n n-r_i)\ge r_i(n-11\ep_n n)	.\label{eq:small:2}
\end{equation}

Applying Lemma \ref{lm:nonzerointerval:nonzeroboxes:stronger} together with the equations \eqref{eq:small:1} and \eqref{eq:small:2}, we have
\begin{equation}
\left (2w+1\right )n^{2}\ge 2N(y)\ge (|\Psi_1|+\dots+|\Psi_{n-1}|)n \ge \sum_{i\in \mathcal S} r_i(n-11 \ep_n n)n,\nonumber
\end{equation}
which gives that
\begin{equation}
\sum_{i\in \mathcal S} r_i\le \frac{2w+1}{1-11\ep_n}\le 2w + 44 w \ep_n + 2.\label{eq:sum:s}
\end{equation}

\quad \textbf{Step 2.2: Realizations of small rows.}
 
\quad \quad \underline{\textit{Choosing the $r_i$.}} The number of sequences $(r_i)_{i\in \mathcal S}$ of nonnegative integers satisfying \eqref{eq:sum:s} is at most the number of nonnegative integer solutions to $r_1+r_2+\dots+r_{n-1} = a$ for some $a\le 2w + 44 w \ep_n + 2$ and is thus bounded by
 \begin{equation}
 \sum_{a=0}^{2w + 44w \ep_n + 2} {{a+n-2}\choose {n-2}}\le 7w {{6w+2n}\choose n}\le 7w \left (\frac{6ew+2en}{n}\right )^{n}\le n^{24 \ep_n w},\label{eq:choose:s}
 \end{equation}
 where we used the assumptions that $a = \ep_n \sqrt{\log n}\ge 1$ and $w\ge \ep_n n$.

\quad \quad  \underline{\textit{Realizations of odd rows $i$.}} For each choice of the set $\mathcal B$, the realizations of rows $j$ with $j\in \mathcal B$, and the sequence $(r_i)_{i\in \mathcal S}$, we need to bound the number of realizations of rows $i$ with $i\in \mathcal S$.

If $i\in \mathcal S$ is odd, then by the definition of skeletons, the number of realizations of row $i$ with a given $r_i$, the size of its skeleton, is at most 
\begin{equation}
{n\choose r_i} q^{2r_i}.\label{eq:choose:odd}
\end{equation}

\quad \quad  \underline{\textit{Realizations of even rows $i$.}}  Having chosen all of the rows $i'$ with $i'$ being odd, for each even index $i\in \mathcal S$, note that the set $\mathcal L_i$, defined earlier in the proof, is fixed as the rows $i\pm 1$ (odd) have been chosen. The number of choices for the restricted skeleton $\tilde{\mathcal I_i}$ is at most $n\choose r_i$. The number of realizations of row $i$ inside the set $\mathcal L_i$ with a given $\tilde{\mathcal I_i}$  is at most $q^{2r_i}$. The number of realizations of row $i$ inside the set $\mathcal L_i^{c}$ is at most $q^{|\mathcal L_i^{c}|}\le q^{5r_{i-1}+5r_{i+1}}$ by \eqref{eq:bound:L}. Thus, the number of realizations of row $i$ given $r_i, r_{i\pm 1}$ and the rows with odd indices is at most
\begin{equation}
{n\choose r_i} q^{2r_i+5r_{i-1}+5r_{i+1}}\label{eq:choose:even}
 \end{equation} 

\quad \quad  \underline{\textit{Combining.}} From \eqref{eq:choose:odd}, and \eqref{eq:choose:even}, the number of realizations of rows $i$ with $i\in \mathcal S$ is at most
\begin{align}
&\prod_{i\in \mathcal S, i=2k+1} {n\choose r_i} q^{2r_i}\prod_{i\in \mathcal S, i=2k}{n\choose r_i} q^{2r_i+5r_{i-1}+5r_{i+1}}\le \nonumber \\
&  \exp \left(  \log n \sum_{i\in \mathcal S}  r_i + 12 \log q \sum_{i\in \mathcal S}  r_i \right) \le n^{2w+45\ep_n w + 160 w\ep_n (\log q)/a} \label{last}
\end{align}

where in the last inequality, we used \eqref{eq:sum:s}.

Combining \eqref{eq:choose:bigrow}, \eqref{eq:choose:s} and \eqref{last},  we obtain the stated lemma.
\end{proof}

Finally, we are ready to prove \eqref{eq:final} which completes the proof of the upper bound of Theorem \ref{main2}.
\begin{lemma}\label{lm:upperbound}
Let $\alpha_n = \log \log (16n) \sqrt{\log n}$. Let $c$ be any number satisfying $c \ge \frac{640}{\log \log (16n)}$. Let $t = \frac{n^2}{4\left (1-\cos \frac{2\pi}{q}\right )} \log n + \frac{n^2}{\left (1-\cos \frac{2\pi}{q}\right )}c \alpha_n\log q$. We have for all $n\ge 3$ and $q\ge 2$, 
\begin{equation}
\Sigma = \sum \left( 1- \frac{(1-\cos 2\pi/q)N(y)}{{n \choose 2}^2} \right)^{2t} \le q^{-2c}/2.\label{eq:upperbound}
\end{equation}
where the sum is taken over all matrices $y\neq 0$ with $N(y) \le \frac{{n \choose 2}^2}{1-\cos \frac{2\pi}{q} }$.
\end{lemma}
\begin{proof}
Let $\ep_n = \frac{1}{100\sqrt{\log n}}$ chosen with foresight. 
We break up the left-hand side of \eqref{eq:upperbound} into sums $T_1, T_2$ where 
\begin{itemize}
\item $T_1$ is the sum over all $y$ with $1\le N(y)\le \frac{n^{2}}{100\ep_n}$, 
\item $T_2$ is the sum over all $y$ with $\frac{n^{2}}{100\ep_n}\le N(y)\le \frac{{n \choose 2}^2}{1-\cos 2\pi/q}$.
 \end{itemize}

Observe that for these matrices $y$, 
$$0\le 1- \frac{(1-\cos 2\pi/q)N(y)}{{n \choose 2}^2}\le \exp\left (-\frac{4(1-\cos 2\pi/q) N(y)}{n^{4}}\right ).$$

{\it Bound $T_1$.} Note that by combining Lemmas \ref{lm:nonzerointerval:skeleton} and \ref{lm:nonzerointerval:nonzeroboxes}, every nonzero matrix $y$ has  $N(y)\ge (n-1)^{2}\ge (1-1/2) n^{2}$. Thus $T_1 \le T_{1, 1} + T_{1, 2}$ where
\begin{equation} 
T_{1, 1} = \sum_{w=1}^{1/(100 \ep_n)} \left |\left \{ y: \left (w-\frac{1}{2}\right )n^{2} < N(y)\le \left (w-40 w \ep_n\right )n^{2}\right \}\right | \exp\left (-\frac{8(1-\cos 2\pi/q)t \left (w-\frac{1}{2}\right )}{n^{2}}\right )\nonumber,
\end{equation}
\begin{equation}\label{key}
T_{1, 2}=\sum_{w=1}^{1/(100 \ep_n)} \left |\left \{ y: \left (w-40 w \ep_n\right )n^{2} < N(y)\le \left (w+\frac{1}{2}\right )n^{2}\right \}\right | \exp\left (-\frac{8(1-\cos 2\pi/q)t \left (w- 40 w \ep_n\right )}{n^{2}}\right )\nonumber.
\end{equation}

Since $\ep_n\le \frac{1}{80}$, applying the second part of Lemma \ref{lm:nonzeroboxes:small}, we deduce that
\begin{eqnarray}
T_{1, 1}&\le& \sum_{w=1}^{1/(100 \ep_n)} q^{16 w}n^{2w-2}  \exp\left (-\frac{8(1-\cos 2\pi/q)t \left (w-\frac{1}{2}\right )}{n^{2}}\right )\le  \sum_{w=1}^{\infty}  q^{(16 -4c\alpha) w } \le \frac{q^{-2c}}{8}\nonumber
\end{eqnarray}
and
 \begin{eqnarray}
T_{1, 2}&\le& \sum_{w=1}^{1/(100 \ep_n)} q^{32 w}n^{2w}  \exp\left (-\frac{8(1-\cos 2\pi/q)t \left (w-40 w \ep_n\right )}{n^{2}}\right )\le \sum_{w=1}^{\infty}  q^{(32 -c\alpha/4) w } \le \frac{q^{-2c}}{8}\nonumber
\end{eqnarray}
where we use the fact that $\alpha\ge 1$ and $c\alpha\ge 640\sqrt{\log n}$.
Thus, $T_1\le q^{-2c}/4$.

{\it Bound $T_2$.} 
We have $T_2 \le T_{2, 1}+T_{2,2}$ where
\begin{eqnarray}
T_{2, 1}= \sum_{w=1/(100 \ep_n)}^{\ep_n n} \left |\left \{ y: \left (w-\frac{1}{2}\right )n^{2} < N(y)\le \left (w+\frac{1}{2}\right )n^{2}\right \}\right | \exp\left (-\frac{8(1-\cos 2\pi/q)t \left (w-\frac{1}{2}\right )}{n^{2}}\right )\nonumber,
\end{eqnarray}
\begin{equation}\label{key}
T_{2, 2}= \sum_{w=\ep_n n}^{\infty} \left |\left \{ y: \left (w-\frac{1}{2}\right )n^{2} < N(y)\le \left (w+\frac{1}{2}\right )n^{2}\right \}\right | \exp\left (-\frac{8(1-\cos 2\pi/q)t \left (w-\frac{1}{2}\right )}{n^{2}}\right )\nonumber,
\end{equation}

Since $\ep_n\le \frac{1}{80}$, applying the first part of Lemma \ref{lm:nonzeroboxes:small} to $T_{2,1}$ yields
\begin{eqnarray}
T_{2, 1}\le \sum_{w=1/(100 \ep_n)}^{\infty} q^{32w}n^{2w+1+60 \ep_n w}  \exp\left (-\frac{8(1-\cos 2\pi/q)t \left (w-\frac{1}{2}\right )}{n^{2}}\right )\le   \frac{q^{-2c}}{8}\nonumber
\end{eqnarray}
 where we used the fact that $c\alpha\log q \ge 320 \ep_n \log n$.

Applying the first part of Lemma \ref{lm:nonzeroboxes:large} to $T_{2,2}$ yields
\begin{eqnarray}
T_{2, 1}\le \sum_{w=\ep_n n}^{\infty} n^{2w+C_1\ep_n w}  \exp\left (-\frac{8(1-\cos 2\pi/q)t \left (w-\frac{1}{2}\right )}{n^{2}}\right )\le \frac{q^{-2c}}{8}\nonumber. 
\end{eqnarray}
where $C_1$ is as in Lemma \ref{lm:nonzeroboxes:large} and we used the fact that $C_1\ep_n \le \frac{300\log q}{\sqrt{\log n}}$ and that $c\alpha\log q \ge 640\sqrt{ \log n}$. 
 Thus, $T_2\le q^{-2c}/4$ which together with the bound on $T_1$ complete the proof.
 \end{proof}

\bibliographystyle{plain}
\bibliography{contingency5}
\end{document}